\documentclass{llncs}

\usepackage[cp1251]{inputenc}
\usepackage[english]{babel}

\usepackage{amsmath,amssymb}
\usepackage{cite}
\usepackage{hyperref}
\usepackage{color}

\newcommand{\RR}{{\mathbb R}}

\newcommand{\ZZ}{{\mathbb Z}}

\newcommand{\no}[1]{^{(#1)}}
\newcommand{\ceil}[1]{\left\lceil #1 \right\rceil}
\newcommand{\floor}[1]{\left\lfloor #1 \right\rfloor}
\newenvironment{narrowarray}[1]{\arraycolsep=0.15em\begin{array}{#1}}{\end{array}}
\newcommand{\fracc}[2]{\frac{\textstyle \strut #1 \strut}{\textstyle \strut #2 \strut}}

\newcommand{\set}[1]{\left\{ #1\right\}}

\begin{document}

\title{On the Proximity of the Optimal Values of~the~Multi-Dimensional Knapsack Problem with~and~without~the~Cardinality Constraint\thanks{This work was performed at UNN Scientific and Educational Mathematical Center.}}

\titlerunning{On the Proximity of the Optimal Values of the Knapsack Problem with and without the Cardinality Constraint}  % abbreviated title (for running head)

\author{A.\,Yu.\,Chirkov\orcidID{0000-0001-5467-8667} 
    \and D.\,V.\,Gribanov\orcidID{0000-0002-4005-9483}
    \and N.\,Yu.\,Zolotykh\orcidID{0000-0003-4542-9233}}

\institute{Lobachevsky State University of Nizhny Novgorod, Gagarin ave. 23, \\
Nizhny Novgorod 603600, Russia,\\
\email{\{aleksandr.chirkov, dmitry.gribanov, nikolai.zolotykh\}@itmm.unn.ru}}

\maketitle

\begin{abstract}
We study the proximity of the optimal value of the $m$-dimen\-sional knapsack problem to the optimal value of that problem with the additional restriction that only one type of items is allowed to include in the solution.    
We derive exact and asymptotic formulas for the precision of such approximation, i.e. for the infinum of the ratio of the optimal value for the objective functions of the problem with the cardinality constraint and without it. In particular, we prove that the precision tends to $0.59136\dots/m$ if $n\to\infty$ and $m$ is fixed. Also, we give the class of the worst multi-dimensional knapsack problems for which the bound is attained.    
Previously, similar results were known only for the case $m=1$.

\end{abstract}

\section{Introduction}

In \cite{KohliKrishnamurti1992,KohliKrishnamurti1995,CapraraKellererPferschyPisinger2000,ChirkovShevchenko2006} the proximity of the optimal value of the (one-dimensional) knapsack problem to the optimal value of the problem with the cardinality constraints was studied. The cardinality constraint is the additional restriction that only $k$ type of items is allowed to include in the solution (i.e. that only $k$ coordinates of the optimal solution vector can be non-zero). Different upper and lower bounds for the guaranteed precision, i.e. for the infinum of the ratio of the optimal value for the objective functions of the problem with the cardinality constraints and without them, were obtained. Also, in some cases the classes of worst problems were constructed.

The importance of such kind of research is due to the fact that some algorithms for solving the knapsack problems require to find an optimal solution to that problem with the cardinality constraints; see, for example \cite{KohliKrishnamurti1992,KohliKrishnamurti1995}, where this approach is used for constructing greedy heuristics for the integer knapsack problem. Moreover, the results of research can be potentially useful for constructing new fully polynomial approximation schemes.

Here, from this point of view, we consider the $m$-dimensional knapsack problem. The solution to that problem with the additional constraint that only $1$ coordinate can be non-zero is called the approximate solution. We derive exact and asymptotic formulas for the precision of such approximation. In particular, we prove that the precision tends to $0.59136\dots/m$ if $n\to\infty$ and $m$ is fixed. Also, we give a class of worst multi-dimensional knapsack problems for which the bound is attained.

\section{Definitions}

Denote by $\ZZ_+$, $\RR_+$ the sets of all non-negative integer and real numbers respectively.
Let
$$
L(A, b) = \set{x\in\ZZ_+^n:~ Ax\le b}, \qquad
A=(a_{ij})\in\RR_+^{m\times n}, \qquad
b=(b_i) \in\RR_+^m.
$$

The {\em integer} $m${-dimensional knapsack problem} is to find $x$
such that
\begin{equation}
cx \to \max\qquad
\text{s.t. }
x\in L(A,b),
\label{f_md_knapsack_problem}
\end{equation}
where $c=(c_j)\in\RR_+^n$ \cite{MartelloToth,KellererPferschyPisinger}.

Denote by $v\no{j}$ ($j=1,2,\dots,n$) a point in $L(A,b)$, all of whose coordinates $v\no{j}_i$ are $0$, except for of $v\no{j}_j$, which is
$$
v\no{j}_j = \min_{i:~a_{ij} > 0} \floor{b_i/a_{ij}}.
$$
It is not hard to see that $v\no{j}\in L(A,b)$ and $cv\no{j}=c_jv\no{j}_j$.
Denote $V(A,b) = \set{v\no{1},\dots,v\no{n}}$.
A point $v\no{j}$, on which the maximum
$$
\max_j cv\no{j}
$$
attained is called an {\em approximate solution} to the problem
(\ref{f_md_knapsack_problem}). 
The {\em precision of the approximate solution} is
$$
\alpha(A,b,c) = \frac{\max\limits_{x\in V(A,b)}cx}{\max\limits_{x\in L(A,b)}cx}.
$$
In this paper we study the value
$$
\alpha_{mn} = \inf_{\substack{A\in\RR_+^{m\times n}\\ b\in\RR_+^m,~ c\in\RR_+^n}} \alpha(A,b,c).
$$

\section{Previous work}\label{sec_Previous_Work}

The precision of the approximate solution to the $1$-dimensional ($m=1$) knapsack problem was studied in \cite{KohliKrishnamurti1992,KohliKrishnamurti1995,CapraraKellererPferschyPisinger2000,ChirkovShevchenko2006}.
In particular, in \cite{KohliKrishnamurti1992,ChirkovShevchenko2006} it was proven that
$$
\delta_n=\delta_{n-1}(\delta_{n-1}+1),
\qquad
\varepsilon_n=1+\varepsilon_{n-1}(\delta_{n-1}+1),
\qquad
\delta_1=\varepsilon_1=1.
$$
The sequence $\set{\delta_n}$ is the A007018 sequence in On-Line Encyclopedia of Integer Sequences (OEIS)~\cite{OEIS}.
The sequence $\set{\epsilon_n}$ is currently absent in OEIS.

The sequence $\alpha_{1n}=\delta_n/\varepsilon_n$ decreases monotonously and tends to the value $\alpha_{1\infty} = 0.591355492056890\dots$
The values for $\delta_n$, $\varepsilon_n$ and $\alpha_{1n}$ for small $n$ are presented in Table~\ref{tab_values}.

\begin{table}[tb]
    \caption{Values of $\delta_n$, $\varepsilon_n$ and $\alpha_{1n}$ for small $n$}\label{tab_values}
    \centering
    \small
    $
    \begin{array}{|c|r|r|l|}
    \hline
    n & \multicolumn{1}{c|}{\delta_n=\delta_{n-1}(\delta_{n-1}+1)} & \multicolumn{1}{c|}{\varepsilon_n=1+\varepsilon_{n-1}(\delta_{n-1}+1)} & \multicolumn{1}{c|}{\alpha_{1n}=\delta_n/\varepsilon_n} \\
    \hline
    1 & 1 & 1 & 1.000000000000000 \\
    2 & 2 & 3 & 0.666666666666667 \\
    3 & 6 & 10 & 0.600000000000000 \\
    4 & 42 & 71 & 0.591549295774648 \\
    5 & 1806 & 3054 & 0.591355599214145 \\
    6 & 3263442 & 5518579 & 0.591355492056923 \\
    7 & 10650056950806 & 18009568007498 & 0.591355492056890 \\
    8 & 113423713055421844361000442 & 191802924939285448393150887 & 0.591355492056890 \\
    %9 & 12864938683278671740537145998360961546653259485195806 & 21754999921504126977590785836876485101156736372842942 & 0.591355492056890 \\
    \hline
    \end{array}
    $
\end{table}

In \cite{KohliKrishnamurti1992,KohliKrishnamurti1995} these results are used in constructing the approximate scheme for the integer knapsack problem. 
Note that $\alpha_{1n}$ is even higher than the guaranteed precision $0.5$ of the greedy algorithm \cite{MartelloToth}. 

The infinum for $\alpha_{1n}$ is achieved on the problem (the worst case)
$$
\sum_{j=1}^n \frac{x_j}{\delta_j} \to \max
$$
s.t.
$$
\sum_{j=1}^n \frac{x_j}{\delta_j+\mu_n} \le 1,
$$
where $0\le\mu_n<1$ and $\sum\limits_{j=1}^n\fracc{1}{\delta_j + \mu_n} = 1$.
In particular, 
$$
\mu_1 = 1, \quad
\mu_2 = \frac{\sqrt{5}-1}{2} = 0.61803\dots, \quad
\mu_3 = 0.93923\dots, \quad
\mu_4 = 0.99855\dots%, \quad
%\mu_5 = 0.99999\dots
$$
The optimal solution vector to this problem is $(1,1,\dots,1)$ and the optimal solution value is $\varepsilon_n/\delta_n$, whereas the approximate solution vectors are 
$$
(1,0,0\dots,0),\quad
(0,\delta_2,0,\dots,0),\quad
(0,0,\delta_3,\dots,0),\quad
\dots,\quad
(0,0,0,\dots,\delta_n)
$$
and the corresponding value of the objective function is $1$.

Lower and upper bounds for the guaranteed precision for $k\ge 2$ are obtained in \cite{ChirkovShevchenko2006}.

In this paper we obtain formulas for $\alpha_{mn}$ for $m \ge 1$.
In particular, we prove that 
$\alpha_{mn}\to\fracc{\alpha_{1\infty}}{m}$ if $n\to\infty$ and $m$ is fixed.

\section{Preliminaries}\label{sec_preliminaries}

\begin{lemma}\label{lemma_monotonicity}
    For any fixed $m$ the sequence $\set{\alpha_{mn}}$ decreases monotonously.
\end{lemma}

\begin{proof}
    Let $A\in\RR_+^{m\times n}$, $h,b\in\RR_+^m$, $c\in\RR_+^n$ and $h>b$.    
    Consider a matrix $A'=(A \mid h)\in\RR_+^{m\times (n+1)}$ and a vector $c' = (c, 0)\in\RR_+^{n+1}$.
    It is not hard to see that all points in $L(A',b)$ are obtained from the points in $L(A,b)$ by writing the zero component to the end.
    Hence $\alpha(A,b,c)=\alpha(A',b,c')\ge \alpha_{m,n+1}$.
    Due to the arbitrariness of $A$, $b$, $c$, we get
    $\alpha_{mn} \ge \alpha_{m,n+1}$.
\end{proof}

\begin{lemma}\label{lemma_A'}
    $\alpha(A,b,c) = \alpha(A',b,c)$ for some $A' \le A$, where each column of $A'$ contains at least one non-zero element.
\end{lemma}

\begin{proof}
Let for some $s$, $t$ we have $a_{st} > 0$ and for all $i\ne s$
$$
\floor{\frac{b_s}{a_{st}}} \le \floor{\frac{b_i}{a_{it}}}
$$
(if there are no such $s$, $t$, then put $A'=A$ and $A'$ has the  required form).
From the matrix $A$ we construct a matrix $A'$ by setting $a'_{it} = 0$ for all $i\ne s$ and $a'_{ij}=a_{ij}$ otherwise.

For all $x\in\RR_+^n$ we have $A'x \le Ax$. Hence $L(A,b)\subseteq L(A',b)$. Hence
$$
\max_{x\in L(A,b)} cx \le \max_{x\in L'(A,b)} cx.
$$
But
$$
\min_{k:~a_{kj}>0} \floor{\frac{b_k}{a_{kj}}} = \min_{k:~a'_{kj}>0} \floor{\frac{b_k}{a'_{kj}}}
\qquad (j=1,2,\dots,n),
$$
hence $V(A,b) = V(A',b)$.
Now we have
$$
\alpha(A,b,c) 
= \frac{\max\limits_{x\in V(A,b)} cx}{\max\limits_{x\in L(A,b)} cx} 
\ge \frac{\max\limits_{x\in V(A',b)} cx}{\max\limits_{x\in L(A',b)} cx}
= \alpha(A',b,c). 
$$

To complete the proof we note that the procedure described above can be performed until the matrix $A'$ acquires the required form.
\end{proof}  

From Lemma~\ref{lemma_A'} it follows that to study $\alpha_{mn}$ it is enough to consider only multi-dimensional knapsack problems with constraints
$$
\newcommand{\pplus}{\!+...+\!}%{\!+\!...\!+\!}
\left\{
\begin{narrowarray}{ccccc}
a_{11}x_1\pplus a_{1l_1}x_{l_1} & & & & \le b_1, \\
& \!\! a_{2,l_1+1}x_{l_1+1}\pplus a_{2,l_2}x_{l_2} & & & \le b_2, \\
\multicolumn{5}{c}{\dotfill} \\
& & & \!\!\!\! a_{m,l_{m-1}+1}x_1\pplus a_{mn}x_{n} & \le b_m, \\
\end{narrowarray}
\right.
%\label{f_direct_product}
$$
that can be called a {\em direct product of $m$ knapsack problems}.
All inequalities $0\le b_i$ have to be deleted due to Lemma~\ref{lemma_monotonicity}.
Denote $n_i = l_i - l_{k-1}$, where $l_0=0$, $l_m=n$ ($i=1,2,\dots,m$).  
Thus, we have proved the following.

\begin{lemma}\label{lemma_direct_product}
    For each $m$, $n$ the infimum $\alpha_{mn}$ is attained on the direct product of knapsack problems. 
\end{lemma}

%$$
%\left\{
%\begin{array}{l}
%a_{11}x_1 + \dots +a_{1n_1}x_{n_1} \le b_1, \\
%a_{2,n_1+1}x_{n_1+1} + \dots + a_{2,n_1+n_2}x_{n_1+n_2} \le b_2, \\
%\multicolumn{1}{c}{\dotfill} \\
%a_{m,n_1+\dots+n_{m-1}+1}x_1 + \dots +a_{mn}x_{n} \le b_m, \\
%\end{array}
%\right.
%$$

%$$
%\left\{
%\begin{narrowarray}{llcccllcl}
%a_{11}           & x_1       & + & \dots & + & a_{1t_1}  & x_{t_1} & \le & b_1, \\
%a_{2,t_1+1}      & x_{t_1+1} & + & \dots & + & a_{2,t_2} & x_{t_2} & \le & b_2, \\
%\multicolumn{9}{c}{\dotfill} \\
% a_{m,t_{m-1}+1} & x_1       & + & \dots & + & a_{mn}    & x_{n}   & \le & b_m. \\
%\end{narrowarray}
%\right.
%$$

\section{The main result}

The main result of the paper is formulated in the following theorem.

\begin{theorem}
For each $m$, $n$
\begin{equation}
    \alpha_{mn} = \fracc{\alpha_{1q}}{m+r\left(\fracc{\alpha_{1q}}{\alpha_{1,q+1}} - 1\right)},
    \label{f_alpha_mn}
\end{equation}
    where $n=qm+r$, $q = \floor{n/m}$.
\end{theorem}

The theorem follows from two lemmas below.

\begin{lemma}\label{lemma_lower_bound}
For each $m$, $n$    
    $$
    \alpha_{mn} \ge \fracc{\alpha_{1q}}{m+r\left(\fracc{\alpha_{1q}}{\alpha_{1,q+1}} - 1\right)}.
    $$
\end{lemma}

\begin{proof}
Thanks to Lemma~\ref{lemma_direct_product}, it is enough to consider only direct products of $m$ knapsack problems. Let $\tau_i = \gamma_i/\beta_i$ 
%$$
%\tau_i = \frac{\gamma_i}{\beta_i} \qquad (i=1,2,\dots,m)
%$$
be the precision of approximate solution to the $i$-th knapsack problem
$(i=1,2,\dots,m)$,
where $\gamma_i$ is the approximate solution value, $\beta_i$ is the optimal solution value. For their product we have
$$
\alpha(A,b,c) 
= \frac{\max\limits_{i=1,\dots,m}\gamma_i}{\sum\limits_{i=1}^m\beta_i}
= \frac{\gamma_s}{\sum\limits_{i=1}^m\beta_i}
= \frac{1}{\sum\limits_{i=1}^m\fracc{\beta_i}{\gamma_s}}
= \frac{1}{\sum\limits_{i=1}^m\fracc{\gamma_i}{\gamma_s\tau_i}}
\ge \frac{1}{\sum\limits_{i=1}^m\fracc{1}{\tau_i}}.
$$
The inequality turns into equality if and only if $\gamma_1=\gamma_2=\dots=\gamma_m$.
Since $\tau_s\ge\alpha_{1n_1}$ then
$$
\alpha(A,b,c)  \ge \frac{1}{\sum\limits_{i=1}^m\fracc{1}{\alpha_{1n_i}}}.
$$

Thus, we obtain the problem to find $n_1,n_2,\dots,n_m$ such that
\begin{equation}
\fracc{1}{\sum\limits_{i=1}^m\fracc{1}{\alpha_{1n_i}}} \to\min
\qquad
\text{s.t. } \sum\limits_{i=1}^m n_i = n.
\label{f_aux_problem}
\end{equation}
The sequence
$$
\frac{1}{\alpha_{1,n+1}} - \frac{1}{\alpha_{1n}} 
= \frac{\varepsilon_{n+1}}{\delta_{n+1}} - \frac{\varepsilon_n}{\delta_n} 
= \frac{1+\varepsilon_n(\delta_n+1)}{\delta_{n+1}} - \frac{\varepsilon_n(\delta_n+1)}{\delta_{n+1}}
= \frac{1}{\delta_{n+1}}
$$
decreases monotonously as $n\to\infty$, hence
%$$
%\frac{1}{\alpha_{1,n+2}} - \frac{1}{\alpha_{1,n+1}}
%\le \frac{1}{\alpha_{1,n+1}} - \frac{1}{\alpha_{1n}},
%$$
%therefore
$$
\frac{1}{\alpha_{1,n+2}} + \frac{1}{\alpha_{1n}} \le \frac{2}{\alpha_{1,n+1}}. 
$$
We conclude that the minimum for (\ref{f_aux_problem}) is reached if
$n_1=\dots = n_r = q+1$, $n_{r+1} = \dots = n_m = q$.
Thus,
$$
\alpha(A,b,c)\ge\fracc{1}{\sum\limits_{i=1}^m \fracc{1}{\alpha_{1n_i}}}
= \fracc{1}{\fracc{r}{\alpha_{1,q+1}} + \fracc{m-r}{\alpha_{1q}}}
= \fracc{\alpha_{1q}}{m+r\left(\fracc{\alpha_{1q}}{\alpha_{1,q+1}} - 1\right)}.
$$
\end{proof}

In the following lemma we construct a class of (worst) multi-dimensional knapsack problems on which the bound (\ref{f_alpha_mn}) is attained.

\begin{lemma}\label{lemma_upper_bound}
For each $m$ and $n$  
$$
\alpha_{mn} \le \fracc{\alpha_{1q}}{m+r\left(\fracc{\alpha_{1q}}{\alpha_{1,q+1}} - 1\right)},
$$
where $n=qm+r$, $q = \floor{n/m}$.
\end{lemma}    

\begin{proof}
Consider the direct product of $r$ knapsack problems of the form
$$
\max \sum_{j=1}^{q+1}\frac{x_j}{\delta_j}
\to \max
\qquad
\text{s.t. }
%\quad
\sum_{j=1}^{q+1}\frac{x_j}{\delta_j+\mu_{q+1}}\le 1
$$
and $m-r$ knapsack problems of the form
$$
\max \sum_{j=1}^q\frac{x_j}{\delta_j} \to \max
\qquad
\text{s.t. }
%\quad
\sum_{j=1}^q\frac{x_j}{\delta_j+\mu_q}\le 1.
$$
The precision of the approximate solutions to these problems is $\alpha_{1q}$ and $\alpha_{1,q+1}$ respectively (see Section~\ref{sec_Previous_Work}).
For the product of these problems the optimal solution value is 
$$
r\, \frac{\varepsilon_{q+1}}{\delta_{q+1}} +
(m-r)\, \frac{\varepsilon_q}{\delta_q} = 
\frac{r}{\alpha_{1,q+1}} +
\frac{m-r}{\alpha_{1q}}
$$
and the approximate solution value is $1$,
hence the precision of the approximate solution is
$$
\alpha(A,b,c) = \fracc{1}{\fracc{r}{\alpha_{1,q+1}} + \fracc{m-r}{\alpha_{1q}}} = \fracc{\alpha_{1q}}{m+r\left(\fracc{\alpha_{1q}}{\alpha_{1,q+1}} - 1\right)}.
$$
\end{proof}

\begin{corollary}\label{cor_alpha_mn_2_ineq}
$$
\fracc{\alpha_{1,\ceil{n/m}}}{m}
\le \alpha_{mn} 
\le \fracc{\alpha_{1,\floor{n/m}}}{m}. 
$$
\end{corollary}    
\begin{proof}
The first inequality obviously follows from (\ref{f_alpha_mn}).
Let us prove the second one.
If $r=0$ then 
$$
\alpha_{mn} = \fracc{\alpha_{1q}}{m} = \fracc{\alpha_{1,\ceil{n/m}}}{m}.
$$
If $0<r<m$ then 
$$
\alpha_{mn} = \fracc{\alpha_{1q}}{m+r\left(\fracc{\alpha_{1q}}{\alpha_{1,q+1}} - 1\right)} 
> \fracc{\alpha_{1q}}{m+m\left(\fracc{\alpha_{1q}}{\alpha_{1,q+1}} - 1\right)} 
= \fracc{\alpha_{1,q+1}}{m}
= \fracc{\alpha_{1,\ceil{n/m}}}{m}.
$$
\end{proof}

From Corollary~\ref{cor_alpha_mn_2_ineq} we obtain the following.
\begin{corollary}
If $n\to\infty$, $m=o(n)$ then $\alpha_{mn}\sim\fracc{\alpha_{1,\floor{n/m}}}{m}$.
\end{corollary}
    
\begin{corollary}
    If $n\to\infty$ and $m$ is fixed then
    $\alpha_{mn}\to\fracc{\alpha_{1\infty}}{m}$.
\end{corollary}

%$$
%\frac{\alpha_{1q}}{\alpha_{1,q+1}} - 1 
%= \frac{\delta_q \varepsilon_{q+1}}{\varepsilon_q\delta_{q+1}}
%= \frac{\delta_q(1+\varepsilon_q(\delta_q+1))}{\varepsilon_q\delta_q(\delta_q+1)}-1
%= \frac{1}{\varepsilon_q(\delta_q+1)} = \frac{1}{\varepsilon_{q+1}-1}
%$$
%which tends to $0$ as $q\to\infty$.

\section{Conclusion} \label{sec_conclusion}

In this paper we derived exact and asymptotic formulas for the precision of approximate solutions to the $m$-dimensional knapsack problem. In particular, we proved that the precision tends to $0.59136\dots/m$ if $n\to\infty$ and $m$ is fixed. The proof of the attainability of the obtained bounds for the precision is constructive.

In the future, our results can be base for new fully polynomial time approximation schemes.

%\paragraph{Acknowledgements}
%This work was supported by the Russian Science Foundation Grant No. 17-11-01336.

%\noindent
%Authors:

%\medskip
%\noindent
%Sergey Olegovich Semenov\\
%postgraduate, Lobachevsky State University of Nizhni Novgorod\\
%\url{Nikolai.Zolotykh@itmm.unn.ru}

%\medskip
%\noindent
%Nikolai Yur'evich Zolotykh\\
%Prof., Lobachevsky State University of Nizhni Novgorod\\
%\url{Nikolai.Zolotykh@itmm.unn.ru}

%\bigskip

%\bibliographystyle{splncs03}

%\bibliography{skeleton}

\end{document}